\documentclass{article}
\usepackage{graphicx}
\usepackage{framed}
\usepackage[normalem]{ulem}
\usepackage{amsmath}
\usepackage{amsthm}
\usepackage{amssymb}
\usepackage{amsfonts}
\usepackage{enumerate}
\usepackage[utf8]{inputenc}
\usepackage[top=1 in,bottom=1in, left=1 in, right=1 in]{geometry}
\usepackage{xcolor}
\usepackage{hyperref}
\usepackage{tikz}
\usepackage[procnames]{listings}
\usepackage[toc]{appendix}

\usetikzlibrary{arrows}

\tikzset{
	vertex/.style={circle,draw,minimum size=1.5em},
	edge/.style={->,> = latex'}
}

\newcommand{\N}{\mathbb{N}}

\newcommand{\Z}{\mathbb{Z}}

\newcommand{\Zm}{\mathbb{Z} / m \mathbb{Z}}

\newtheorem{theorem}{Theorem}[section]
\newtheorem{corollary}[theorem]{Corollary}
\newtheorem{lemma}[theorem]{Lemma}
\newtheorem{prop}[theorem]{Proposition}

\theoremstyle{definition}
\newtheorem{definition}[theorem]{Definition}

\newtheorem{remark}[theorem]{Remark}

\title{On structures induced by the power sequences of $(\Zm, \cdot)$\thanks{The publication is unclassified and has been assigned LANL identifier LA-UR-20-21570.} }
\author{Kelly Isham \thanks{ The first author was supported in part by an appointment with the National Science Foundation (NSF) Mathematical Sciences Graduate Internship (MSGI) Program sponsored by the NSF Division of Mathematical Sciences.}
	\and Laura Monroe \thanks{A portion of this work was performed at the Ultrascale Systems Research Center (USRC) at Los Alamos National Laboratory, supported by the U.S. Department of Energy contract DE-FC02-06ER25750.}
}

\date{}

\begin{document}

\maketitle
\begin{abstract}
In this paper, we explore the structure of $\Zm$ in terms of its orbits under modular exponentiation, illustrating this with a sequential power graph that is naturally derived from the orbits by connecting elements of $\Zm$ in the orbit order in which they appear. 

We find that this graph has a great deal of fascinating algebraic structure. The connected components are composed of orbits that all share at least one element. The vertex sets of the connected components are shown to depend on the factorization of $m$; in fact, the connected components are completely determined by the units of $\Zm$, the idempotents of $\Zm$ and the square-free divisors of $m$. Both tails and non-tails of the components can be described explicitly and algebraically in terms of these elements of $\Zm$. Finally, a lattice of components can be used to show homomorphisms between the non-tails of any two comparable components in the lattice. 

This extensive structure is used here to prove an algebraic identity on the roots of an idempotent mod $m$, and may be exploited to prove other identities as well.

\end{abstract}

\section{Introduction}

The structure of the unit group $(\Zm)^\times$ is well-understood. In this paper, we explore the structures of other subsets of $\Zm$, whether or not they are units, in terms of their orbits under modular exponentiation. 

Consider the power sequence $a, a^2, a^3, \ldots $ for some $a \in \Zm$. Since this monoid is finite, this sequence must be finite as well, that is, there exists $j < k$ so that $a^j = a^k$ in $\Zm$. We call the sequence $a, a^2, \ldots, a^{k-1}$ the \textit{orbit} of $a$. We then break the orbit into two distinct parts: the \textit{cycle} of $a$ is the sequence $a^j, a^{j+1}, \ldots, a^{k-1}$ and the \textit{tail} of $a$ is the sequence $a, a^2, \ldots, a^{j-1}$. (Observe that if $j=1$, then the orbit of $a$ has no tail.) 

In order to understand all orbits in $\Zm$, we group all orbits that share common elements together by defining a graph $(V, E)$. The vertex set is $\Zm$ and the elements $a,b$ are connected by a directed graph if and only if they appear next to each other in some orbit. More formally, $(a,b) \in E$ if and only if $a = c^i$ and $b = c^{i+1}$ for some $c \in \Zm$ and $i \in \N.$ 

We call this graph the \textit{sequential power graph}. The sequential power graph is very structured, as we will see throughout this paper. 

\subsection{Prior work}
While we have not seen this graph defined in the literature, there are two similar graphs that have been discussed that share several properties with the sequential power graph. The \textit{cycle grap}h was defined by Shanks in \cite{shanks}. This graph depicts orbits of elements $a \in \Zm$ that do not overlap. For example, if $m = 15$, then the orbit $3, 9, 12, 1$ will be included, but not $9,1$ since this is a subsequence of the orbit of 3. The edges for this orbit would be $(3,9), (9,12), (12,1),$ and $(1, 3)$. This graph is a subgraph of the sequential power graph and the vertex set for each component is the same in both graphs. In particular, Shanks' graph is the first to illustrate tail elements and Shanks proves some basic properties of tails. Kelarev and Quinn in \cite{kelarev} defined the \textit{directed power graph} with vertex set $\Zm$ and edges $(a,b)$ if and only if $b$ is a power of $a$. Other authors, such as Chakrabarty, Ghosh, and Sen in \cite{chakrabarty}, considered an undirected version of the power graph. 

The set of orbits of elements in $\Zm$ has also been studied from the point of view of semigroup theory. One can define the vertices in each component of the above graphs as the sets $C_d = \{x \in \Zm \, : \, x^k = d\}$ for some idempotent $d$ modulo $m$. Hewitt and Zuckerman in \cite{hewitt} and Schwarz in \cite{schwarz} used this approach to study $C_d$. They found that there is a subgroup of $(\Zm, \cdot)$ contained in each $C_d$ with identity $d$. These are called \textit{maximal groups} of $(\Zm, \cdot)$. Schwarz also provides a homomorphism between the group of units and the other maximal groups.

The maximal subgroups play an important role in the structure of the sequential power graph and several of the ideas from these authors apply here. However, the proofs in both \cite{hewitt} and \cite{schwarz} are algebraic and do not explicitly describe the elements in each $C_d$. 

\subsection{Contributions}
We introduce the concept of a sequential power graph and explore its algebraic structure, emphasizing non-units as well as units. We show that the vertex sets of the connected components of this graph are completely determined by the factorization of $m$. 

We also show the correspondence between the connected components, the units $(\Zm)^\times$, the idempotents of $\Zm$ and the square-free divisors  of $m$: in fact, the structure of  the unit component $(\Zm)^\times$ in combination with a given square-free  divisor  $\pi$ of $m$ suffices to determine the structure of the connected component associated with $\pi$, including its  idempotent. 

We create a lattice of components of the sequential power graph and use this to define a homomorphism that works between any two components that are comparable within the lattice. We also demonstrate that the inverse of this map is easy to understand.

We provide number-theoretic proofs of several previously known results found in the above papers, and then extend these results by classifying the structure of the tails. We also provide formulas for determining the sum of all elements in a component $C$ where the elements are thought of as integers; that is, we take all $b \in C$ where the $b$ are elements of $\Zm$, and add the $b$ as ordinary integers. These formulas depend only on $m$ and the idempotent defining this component. 

\section{Background}

\subsection{Orbits and cycles}
We start with any element $a \in \Zm$ and compute the orbit of $a$. Because the ring $\Zm$ is finite, the orbit stabilizes at some point (i.e., $a^k \equiv a^{\ell}\pmod{m}$ for some $k$ and $\ell\in \N$). In other words, every orbit contains a cycle. 

If the least such $k=1$, then the orbit does not have a tail, so the entire orbit forms a cycle. If not, the orbit does have a tail, and its cycle is a proper subset of the orbit. These cycles may be common to several orbits. 

\subsection{Definitions} \label{definitions}
\subsubsection{General definitions}
Throughout this paper, we assume that the factorization of $m$ is known.

\begin{definition}
	A \textit{sequential power graph} is a digraph with vertex set $\Zm$ and an edge between $b$ and $c$ if they appear sequentially in the set \{$a, a^2, a^3, \ldots$\} of some element $a \in \Zm$, that is, if $a \equiv c^i \pmod{m}$ and $b \equiv c^{i+1} \pmod{m}$ for some $c \in \Zm$ and some $i \in \N$. 
\end{definition}
Note that we do not discuss the graph theoretic properties in this paper. We simply use this expression as a graph as a visual aid and to make the definition about connected components below. For example, when $m = 36$, we have the following graph\\
\begin{center}
	\includegraphics[width=19cm, trim={.9cm 0 0 0},clip]{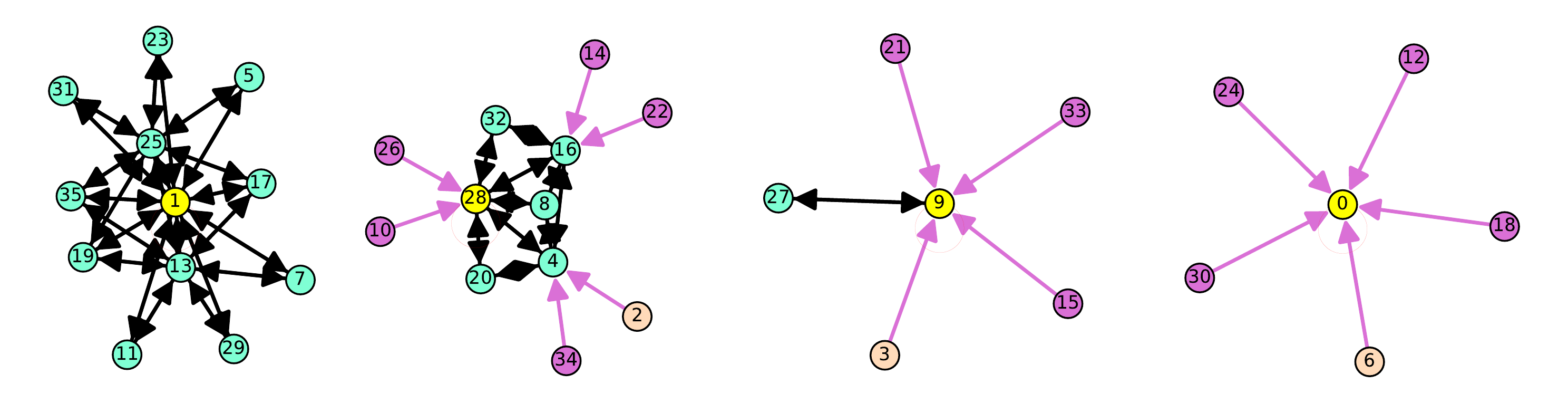}
\end{center}
where the tails are colored pink, the idempotents are yellow, and the multipliers are peach.

\subsubsection{Structural definitions}
\begin{definition}
	An \textit{orbit} of an element $a$ is the set of powers of $a \pmod{m}$ $\{a^1, a^2, \cdots\}$.
\end{definition}
\begin{definition}
	A \textit{cycle} in the orbit of $a$ is a subsequence $a^j, a^{j+1}, \ldots , a^{j+\ell} = a^j$, where $j,\ell\in\N, j<\ell$, and $j$ and $\ell$ are the least such integers having this property.
\end{definition}
\begin{definition}
	A \textit{tail} is the initial sequence $a , a^2 , \ldots, a^{j-1}$ of an orbit, where $j$ is the least such integer so that $a^j$ starts a cycle in the orbit of $a$.
\end{definition}
\begin{definition}
	A \textit{connected component} of a sequential power graph is a maximal set of vertices all of which can be reached from each other by some (undirected) path.
\end{definition}

\subsubsection{Special elements in the connected components}
\begin{definition}
	Let $I = \{i_1, \ldots, i_s\}\subseteq R$. Then $\pi_I = p_{i_1} \cdots p_{i_s}$ is called a \textit{multiplier} of $m$ with respect to $I$. \end{definition}

\begin{definition}
	An element $d \in \Zm$ is \textit{idempotent} if $d^2 \equiv d \pmod{m}$.
\end{definition}

\begin{remark}
	We will prove the structure of the idempotents in Theorem \ref{idem_struc}. Let $I = \{i_1, \ldots, i_s\}$. To introduce notation early on, this theorem implies that $d_I = a_I g_I$ where $g_I = p_{i_1}^{e_{i_1}} \cdots p_{i_s}^{e_{i_s}}$ and $a_I$ can easily be computed from $m$ and $g_I$. The idempotent $d_I$ is determined by the set $I$. The elements $g_I$ are also very interesting and we will refer to them throughout the paper.
\end{remark}

\subsubsection{Other interesting elements}

\begin{definition}
	An element $n \in \Zm$ is \textit{nilpotent} if $n^k \equiv 0 \pmod{m}$ for some $k \in \N$.
\end{definition}

\section{Notation}\label{notation}
We define certain notations that we will use throughout this paper.
\begin{itemize}
	\item $m = p_1^{e_1} \cdots p_r^{e_r}$ is the modulus. 
	\item $R = \{1, 2, \ldots, r\}$ is the set of all indices referring to the prime powers dividing $m$. 
	\item $I = \{i_1, \ldots, i_s\} \subseteq R$ defines a connected component $C_I$, its idempotent and its multiplier, as discussed later in the text. 
	\begin{itemize}
		\item[$\ast$] $g_I=$ gcd$(d_I,m)= p_{i_1}^{e_{i_1}} \cdots p_{i_s}^{e_{i_s}}$ is the greatest common divisor of $d_I$ and $m$.
		\item[$\ast$] $d_I = a_Ig_I$ is the idempotent corresponding to $C_I$, with $a_I \equiv g_I^{-1} \pmod{\frac{m}{g_I}}$.
		\item[$\ast$] $p_I = p_{i_1} \cdots p_{i_s}$ is the multiplier of $C_I$.
	\end{itemize}
	\item $R \setminus I = \{j_1, \ldots, j_{r-s}\} \subset R$ defines the connected component $C_{R \setminus I}$, the dual to $C_I$.
	\begin{itemize}
		\item[$\ast$] $g_{R \setminus I}=$ gcd$(d_{R \setminus I},m)= p_{j_1}^{e_{j_1}} \cdots p_{j_{r-s}}^{e_{j_{r-s}}} = \frac{m}{g_I}$ is the greatest common divisor of $d_{R \setminus I}$ and $m$.
		\item[$\ast$] $d_{R \setminus I} = a_{R \setminus I}g_{R \setminus I}$ is the idempotent corresponding to $C_{R \setminus I}$, with $a_{R \setminus I}$ relatively prime to $\frac{m}{g_{R \setminus I}}=g_I$.
		\item[$\ast$] $p_{R \setminus I} = p_{j_1} \cdots p_{j_{r-s}}$ is the multiplier of $C_{R \setminus I}$.
	\end{itemize}
	
\end{itemize}

\section{The units and idempotents of $\Zm$}
As we will see below, the structure of the sequential power graph of $\Zm$ is entirely defined by the units $U = (\Zm)^\times$, the set of idempotents $d$ of $\Zm$, and the set of multipliers $\pi$. The structure of $U$ and the idempotents has been studied extensively. We discuss in this section the structure of these sets and derive a few useful identities.

\subsection{Structure of the component $U=(\Zm)^\times$}
\label{unit_struc}

The elements in $U$ are the roots of 1, and are thus connected to 1 in the graph. Further, if $a \in \Zm$ were such that $a^k \equiv 1 \pmod{m}$ for some $k \in \N$, then $\text{gcd}(a^k, m) = 1$ and therefore $\text{gcd}(a, m) = 1$. This implies that $U$ forms one of the connected components of $\Zm$. This component is a group under multiplication, and has well-understood structure (see e.g., \cite{ireland}),  which depends on the form of $m$. We restate that structure here, for completeness.

\begin{theorem}
	\cite[Chapter 4, Theorem 3]{ireland} Let $m = 2^e p_1^{e_1} \cdots p_{r}^{e_{r}}$, where all $p_i$ are odd primes. Then
	\[
	(\Zm)^\times \simeq 
	\begin{cases}
	(\Z / p_1^{e_1}\Z)^\times \times (\Z / p_2^{e_2}\Z)^\times\times \cdots \times (\Z / p_r^{e_r}\Z)^\times , & e=0 \text{ or } 1 \\
	\Z / 2 \Z \times (\Z / p_1^{e_1}\Z)^\times \times (\Z / p_2^{e_2}\Z)^\times\times \cdots \times (\Z / p_r^{e_r}\Z)^\times , & e=2 \\
	\Z / 2\Z \times \Z / 2^{e-2} \Z\times (\Z / p_1^{e_1}\Z)^\times \times (\Z /p_2^{e_2}\Z)^\times\times \cdots \times (\Z /p_r^{e_r}\Z)^\times , & e>2 \\
	\end{cases}
	\]
\end{theorem}

\subsection{Structure of the idempotents in $\Zm$}
We record the important structural idempotent theorems here for use later. 
\label{idems}
\begin{prop}
	The idempotents of $\Zm$ form a monoid under multiplication.
\end{prop}

\begin{theorem}\cite[Theorem 2.2]{hewitt}  Let $I = \{i_1, \ldots, i_s\}\subseteq R$. Then the idempotents of $\Zm$ are of the form $a_Ig_I \pmod{m}$ where $a_I \equiv g_I^{-1} \pmod{\frac{m}{g_I}}$.
	\label{idem_struc}
\end{theorem}

Theorem \ref{idem_struc} proves the fact from Section \ref{notation}, that the idempotent $d_I$ is determined by $I$. 

\begin{corollary}\cite[Theorem 2.2]{hewitt}
	There are $2^r$ idempotent elements in $\Zm$.
	\label{num_idem}
\end{corollary}
\begin{proof}
	This follows easily from the proof of Theorem \ref{idem_struc}(\cite[Theorem 2.2]{hewitt}).
\end{proof}

\section{Structure of the connected components}
We discuss here the structure of the connected components of the graph of orbits. As mentioned above, not all elements of the component identified with $d$ need have the form $du$. In other words, $dU$ need not comprise the entire connected component. Those elements of the connected component that are not in $dU$ form tails.

The authors in \cite{chakrabarty} define the power graph $g(\Zm)$ to be the graph with vertex set $\Zm$. Two vertices $a,b$ are connected by an undirected edge if and only if $a \ne b$ and $a^k \equiv b \pmod{m}$ or $b^k \equiv a \pmod{m}$ for some $k \in \N$. While the graph they define is different than our sequential power graph, the vertices that form each connected component are the same. Therefore we use their structure theorem for the vertex set below. 

\begin{theorem}\cite[Theorem 3.9]{chakrabarty}
	\label{comp_struc}
	Let $I \subseteq R$. The component $C_I$ determined by $I$ is of the form
	$$
	C_I \cong p_{i_1} \left(\Z /p_{i_1}^{e_1} \Z\right) \times p_{i_2}\left( \Z / p_{i_2}^{e_2}\Z\right) \times \cdots \times p_{i_s}\left(\Z / p_{i_s}^{e_s} \Z\right) \times U_n
	$$
	where $n = \frac{m}{g_I}$ and $|C_S| = \frac{\phi(m)}{\phi(\pi_I)}.$ \end{theorem}

Further, we can explicitly describe the elements in each component $C_I$.

\begin{prop}
	$C_I = \{\pi_Ix \pmod{m} \, : \,  \text{gcd}(x, \frac{m}{g_I}) = 1\}$.
\end{prop}
\begin{proof}
	Recall that there is a unique idempotent element in every cycle. If $v = \pi_I x$, then $v^k =d$ in $\Zm$ for some $k$. Therefore $\pi_I^k x^k = d$. Since
	the only primes dividing $\pi_I$ are those with index set $I$ and since $\text{gcd}(x^k , \frac{m}{g_I}) =1$, then $d= d_I$. Now suppose there were some other element $w \in C_I$ not of the form $\pi_Ix$. Since $w^\ell = d_I$ for some $\ell \in \N$, we see that the only primes dividing $w$ are the $p_i$ such that $i \in I$. Therefore we can write $w  = p_{i_1}^{j_1} \cdots p_{i_s}^{j_s} y$ where $I = \{i_1, \ldots, i_s\}$, $j_k \in \N$, and $\text{gcd}(y, \frac{m}{g_I})= 1$. Notice we can express $w = p_{i_1} \cdots p_{i_s} x = \pi_I x$ with $ x =p_{i_1}^{j_1-1} \cdots p_{i_s}^{j_s-1} y$ satisfying the condition $\text{gcd}(m , \frac{m}{g_I}) = 1$.
\end{proof}

\subsection{Structure of all subcomponents $dU$ ($d$ an idempotent)}
We discuss here idempotent multiples of the group of units $U$. Much of the structure of the sequential power graphs corresponding to $dU$ is determined by $U$.  As we will show in Section \ref{unique_dU}, each of these is in a different connected component. 

Note that $dU$ may but need not comprise the entire connected component. We discuss the case where $dU$ is strictly contained in its component, which thus has at least one tail, in Section \ref{tails}.

Throughout this section, let $U = \big(\Zm)^\times$ and let $d\in \Zm$ be idempotent.

\subsubsection{Structure of $dU$}
\begin{prop} \label{dU_connected} $dU$ is connected. \end{prop}

\begin{proof}
	Let $u$ be any unit, and let $d$ be an idempotent. Let $k$ be the order of $u$, so $u^1, u^2, \ldots , u^k=1$ is the associated cycle. Then $(du)^1, (du)^2, \ldots , (du)^k$ is the same sequence as $du^1, du^2, \ldots , du^k=d$ and so a cycle, with $d$ as its idempotent.  As in the unit case, all of the $du$ are connected by a path to $d$, and thus form a subset of their connected component. 
\end{proof}

\begin{prop}\cite[Theorem 2.3]{schwarz}\label{dU_mult_group}
	$dU$ is a multiplicative group, with identity $d$.
\end{prop}
\begin{proof}
	If $du, dv\in dU$, then $dudv = d^2 uv \equiv duv \pmod{m}$. Therefore $dU$ is closed under multiplication. Further, $(du)(du^{-1}) = d^2 uu^{-1} \equiv d \pmod{m}$, so inverses exist. 
\end{proof}

The next proposition is due to Schwarz, who proved this using algebraic techniques. We record a number-theoretic proof here.
\begin{prop} \cite[Corollary 4.1]{schwarz} Let $g = \text{gcd}(d,m)$. Then $|dU| = \phi(\frac{m}{g}).	$
	\label{size_dU}
\end{prop}
\begin{proof}
	Let $U = \{u_1, \ldots, u_{\phi(m)}\}$. Consider the elements $\{du_1, \ldots, du_{\phi(m)}\}$. Two elements of this set are the same  in $\Zm$ if and only if $du_i \equiv du_j \pmod{m}$ if and only if $\frac{d}{g} u_i \equiv \frac{d}{g} u_j \pmod{\frac{m}{g}}$. Since $g = \text{gcd}(d,m)$ then $\text{gcd}(\frac{d}{g}, \frac{m}{g}) = 1.$ Therefore we can multiply both sides by the inverse of $\frac{d}{g}$ to obtain the expression $u_i \equiv u_j \pmod{\frac{m}{g}}$.  All of these implications are reversible, so $du_i \equiv du_j \pmod{m}$ if and only if $u_i \equiv u_j \pmod{\frac{m}{g}}$. Since $\text{gcd}(u_k,m) =1$ then $\text{gcd}(u_k, \frac{m}{g}) =1$ for all $1\leq k \leq \phi(m)$. Therefore we get distinct elements in $dU$ for every congruence class corresponding to a unit modulo $\frac{m}{g}.$ There are $\phi(\frac{m}{g})$ units modulo $\frac{m}{g}$ and thus $|dU| = \phi(\frac{m}{g}).$
\end{proof}

\begin{theorem} \cite[Lemma 4.3]{schwarz}
	\label{dU_struc}
	Let $d_I$ be an idempotent and let $ J = R \setminus I = \{j_1, \ldots, j_s\}$. Then $d_IU\cong (\Z / p_{j_1}^{e_{j_1}}\Z)^\times \times (\Z/ p_{j_2}^{e_{j_2}}\Z)^\times\times \cdots \times (\Z/ p_{j_s}^{e_{j_s}}\Z)^\times  
	$.
\end{theorem}
\begin{proof}
	First, compute $$d_IU=d_I \cdot (\Zm)^\times=d \cdot \left((\Z / p_1^{e_1}\Z)^\times \times (\Z / p_2^{e_2}\Z)^\times\times \cdots \times (\Z / p_r^{e_r}\Z)^\times\right).$$
	Observe that
	$d_I \cdot \Z / p_i\Z\cong \{0\}$ when $i \in I$, so $$d_IU \cong (\Z / p_{j_1}^{e_{j_1}}\Z)^\times \times (\Z /p_{j_2}^{e_{j_2}}\Z)^\times\times \cdots \times (\Z /p_{j_s}^{e_{j_s}}\Z)^\times.$$
\end{proof}

\begin{corollary}
	Let $I = R \setminus \{i\}$. Then $d_IU \cong (\Z/ p_i^{e_i}\Z)^\times$. In particular, $d_IU$ is cyclic if and only if $p_i$ is odd or $p_i=2$ and $e_i = 1,2$.
\end{corollary}

The previous theorem by Schwarz gives an algebraic description of $dU$; we explicitly write down the elements in $dU$ in the following proposition.

\begin{prop} Let $g = \text{gcd}(d,m)$ and let $S =\{gx \pmod{m}\, : \, \text{gcd}(x, \frac{m}{g} )= 1\}$. Then $dU = S$. \end{prop}
\begin{proof}
	Let $d$ be an idempotent element and let $g = \text{gcd}(d,m)$. Fix $u \in U$ and consider $x \equiv \frac{d}{g} u \pmod{\frac{m}{g}}$. Then $x$ is a unit modulo $\frac{m}{g}$ and $gx \equiv d u \pmod{m}$. Thus $du \in S$.  By Proposition \ref{size_dU}, there are exactly $\phi(\frac{m}{g})$ elements in $|dU|$. Further, $|S| = \phi(\frac{m}{g}).$ Since $dU \subset S$ and $|dU| = |S|$, then $dU = S$. 
\end{proof}

\subsubsection{Algebraic structure of the cross product of the $dU$}
In the next section, we consider the cross product of the cycles $d_IU$. This corresponds to looking at the vertex sets of the components of $d_IU \times d_JU$ viewed as the graph product.
\begin{prop}
	Let $d_I$ and $d_{R \setminus I}$ be dual idempotents. Then $d_I U \times  d_{R \setminus I} U \cong U.$
\end{prop}
\begin{proof}
	Recall that 
	$$
	d_I U \simeq (\Z / p_{j_1}^{e_{j_1}}\Z)^\times \times \cdots \times	(\Z / p_{j_s}^{e_{j_s}}\Z)^\times
	$$
	where $\{j_1, \ldots, j_s\} = R \setminus I$ and 
	$$
	d_{R \setminus I} U \simeq (\Z / p_{k_1}^{e_{k_1}}\Z)^\times \times \cdots \times(\Z /p_{k_t}^{e_{k_t}}\Z)^\times
	$$
	where $\{k_1, \ldots, k_t\} = R \setminus (R \setminus I) = I.$ Therefore $d_I U \times d_{R \setminus I} U \simeq U$. 
\end{proof}

\begin{prop}
	Let $d_{R \setminus \{i\}}$ be the top-level idempotents. Then $d_{R \setminus \{1\}} U \times  \cdots \times d_{R \setminus \{r\}} U \simeq U$. 
\end{prop}
\begin{proof}
	This follows from the fact that
	$$
	d_{R \setminus \{i\}}U \simeq (\Z / p_i^{e_i}\Z)^\times.
	$$
\end{proof}

\begin{prop}
	Let $D = \{d_{J_1}, \ldots d_{J_s}\}$ be the set of all idempotents such that $|J_i| = k$ for all $i = 1, \ldots , s$. Then $d_{J_1} U \times \cdots \times d_{J_s} U \simeq \underbrace{U \times \cdots \times U}_{\binom{r-1}{k} \text{ times}}$.
\end{prop}
\begin{proof}
	Each $d_{J_i} U$ is a cross product of $r -k$ groups of the form $(\Z / p_i^{e_i}\Z)^\times$ for some prime dividing $m$. Further there are $\binom{r}{k}$ idempotents in the set $D$. Thus there are a total of $(r-k) \binom{r}{r-k}$ groups in the resulting cross-product and each of the $r$ primes are used the same number of times. Therefore each group in the cross product of $U$ appears
	$$
	\frac{\binom{r}{k}(r-k)}{r} =  \binom{r-1}{k}
	$$
	times.
	
\end{proof}

\begin{prop}
	Let $J = \{j_1, \ldots, j_s\}$ and let $d_J$ be the idempotent in component $C_J$. Let $\{d_{R \setminus \{i_1\}}, \ldots, d_{R \setminus \{i_s\}}\}$ be the set of idempotents where $\{i_1, \ldots, i_s\} = R \setminus J$. Then $d_{R \setminus \{i_1\}} U \times \cdots d_{R \setminus \{i_s\}} U \simeq d_J U$.
\end{prop}

\begin{proof}
	This follows from the observations that $d_{R \setminus \{i_j\}} \simeq (\Z / p_{i_j}^{e_{i_j}}\Z)^\times$ where $i_j \not \in J$ and that $d_J$ is a cross product of $(\Z / p_i^{e_i}\Z)^\times$ such that $i \not \in J$. 
\end{proof}

\subsection{Each connected component $C$ contains a unique idempotent $d$}
\begin{lemma} 
	\label{one_idem}
	There is exactly one idempotent element in each cycle, and so one idempotent in each orbit.
\end{lemma}

\begin{proof}
	Consider the set of powers of $a$. Let $j$ and $\ell$ be the smallest integers such that $a^j, a^{j+1}, \ldots , a^{j+\ell} \equiv a^j \pmod{m}$. Then $$(a^{j+\ell-1})^2\equiv a^{2j+2\ell-2}=a^{j+\ell}a^{j+\ell-2}\equiv aa^{j+\ell-2} \equiv a^{j+\ell-1} \pmod{m},$$ so $a^{j+\ell-1}$ is idempotent. To show uniqueness, we note that if $a^j$ is the first idempotent occurring in the sequence $a^1, a^2, \ldots , a^j$, then all subsequent powers of $a$ in that sequence must be of the form $a^i, 1 \le i \le j$. So there are no more idempotents in the set of powers of $a$.
\end{proof}

\begin{lemma} 
	\label{component_connection}
	If two orbits share an element, then the sequences also have overlapping cycles.
\end{lemma}
\begin{proof}
	Let $a^k=b^\ell$. The power sequence of this element contains exactly one cycle, and this cycle is a subcycle of both power sequences {$a$} and {$b$}.
\end{proof}

\begin{lemma} 
	\label{share_idem}
	If two orbits share an element, then they share their unique idempotent.
\end{lemma}
\begin{proof}
	Follows from Lemmas \ref{one_idem} and \ref{component_connection}.
\end{proof}

We can now state the main structure theorem relating idempotent elements and connected components $C$. This theorem will allow us to classify all $C$ and will be useful later when we relate $C$ to the unit component $U$. 

\begin{theorem} \label{unique_idem}
	There is exactly one idempotent element in each connected component. \end{theorem}
\begin{proof}
	A set of power sequences makes up a connected component if and only if there is a path from any power sequence in the set to any other in the set. This happens if and only if there are shared elements between power sequences forming the links in the path. So, by Lemma \ref{share_idem}, a connected component has a unique idempotent element.
\end{proof}

\begin{corollary}
	Each connected component $C$ is the set of roots of its idempotent $d$.
\end{corollary}

\begin{theorem} 
	\label{idem_comp}
	There is a one-to-one correspondence between idempotents and the connected components that contain them. We may thus identify connected components with their associated idempotent. 
\end{theorem}
\begin{proof}
	Follows from Theorem \ref{unique_idem} and the fact that the idempotent is its own orbit, so is a subset of some connected component.
\end{proof}

\begin{corollary} 
	There are $2^r$ connected components in the sequential power graph.
\end{corollary}
\begin{proof}
	This follows from Corollaries \ref{num_idem} and \ref{idem_comp}.
\end{proof}

\subsection{Each connected component $C$ has a unique multiplier $\pi$}
Theorem \ref{idem_comp} provides a relationship between connected components $C$ and the idempotent elements $d$. We now provide a different structure theorem that relates $C$ and the multipliers $\pi_I$. We will show later in Section \ref{tails} that knowing $I\subseteq R$ and the unit component $U$ will completely determine the structure of each connected component.  

\begin{theorem} \label{mult_comp}
	Let $C_I$ be a connected component of the sequential power graph, and let $d_I$ be the unique idempotent element in $C_I$. Then $C_I$ contains a unique multiplier $\pi_I$.
\end{theorem}
\begin{proof}
	Let $ c$ be any element in $C$. Then eventually $c^k \equiv d_I\pmod{m}$. Since $\pi_I$ divides $m$ and $d_I$, then it must also divide $c$. Therefore $c$ is a multiple of $\pi_I$. Further, if $p | m$ and $p | c$, then $c^k \equiv d_I \pmod{m}$ implies $p | d_I$. Therefore the only possible prime factors of $c$ that divide $m$ must be those in $\pi_I$. This implies that there is at most one multiplier in each component.
	
	Observe that there are $2^r$ possible multipliers. Since there are $2^r$ connected components and each multiplier must be in a different connected component, then every connected component contains exactly one multiplier. 
\end{proof}

\subsection{Each connected component $C$ contains a unique $dU$} \label{unique_dU}
\begin{theorem} 
	Each connected component $C_I$ with idempotent $d_I$ contains $d_IU$. This is the unique $dU$ in $C_I$, and is the set of all elements in cycles in $C_I$.
\end{theorem}
\begin{proof}
	Let $U$ be the unit component and let $d_I$ be the idempotent determining $C_I$. Let $d_IU$ be the connected component of the cycles obtained by multiplying every vertex label in $U$ by $d_I$. We know from above that $d_IU$ is a connected multiplicative group. Since $1 \in U$, then $d_I\in d_IU$. Therefore $C_I$ contains $d_IU$. Uniqueness follows from the uniqueness of $d_I$.
\end{proof}

\begin{remark}
	We can now show that a component $C$ is fully determined by $I\subseteq R$. We do so by finding a bijective relation between $\pi_I$ and $d_I$.
	Given the idempotent $d_I \in C_I$, it is easy to determine the multiplier for $C_I$. Given the multiplier $\pi_I$, we can find $g_I = p_{i_1}^{e_{i_1}} \cdots p_{i_s}^{e_{i_s}}$. From here, we know that $d \equiv g_I a_I \pmod{m}$ where $a_I \equiv g_I^{-1} \pmod{\frac{m}{g_I}}$.
\end{remark}

\subsection{Tails of components $C$ }
\subsubsection{Classification of the tails}
\label{tails}
We now consider the structure of the tails in each component $C$ of the sequential power graph. In particular, we classify when a component will have tails and we relate the group of units to the tails as well. Note that \cite[Theorem 50]{shanks} is similar to the following theorem, though we make the result more explicit. Further, in \cite{effinger}, Effinger proves this statement in the special case $a \not \equiv a^{\phi(m)+1} \pmod{m}$. 

\begin{theorem}Let $C_I$ be a connected component in the sequential power graph of $\Zm$. 
	Then $v$ is a tail if and only if $g_I \nmid v$. \label{tail_struc}
\end{theorem}

\begin{proof}
	We will prove this proposition by proving that $v$ is in the cycle of the orbit of $v$ if and only if $g_I | v$. Write $d_I = g_Ia_I$ and recall $\text{gcd}(a_I, m) =1$.  Let $v \in C$. 
	
	Suppose in the cycle $\{v^j ,v^{j+1}, \ldots, v^k\}$ we have $i=1$. That is, $v$ is in the cycle of powers of $v$. Then in particular, $v^k \equiv d_I \pmod{m}$. Therefore $v \equiv v^{k} v^{j-k+1} \equiv d_I v^{j-k+1} \pmod{m}$. Since $g | d_I$ and $g | m$, then $g | v$.
	
	Suppose that $g_I | v$. Then $v = g_Ix$ for some $x$ relatively prime to $\frac{m}{g_I}$. Let $k \in \N$ be such that $(g_Ix)^k \equiv 1 \pmod{\frac{m}{g_I}}$. Then $(g_Ix)^{k+1} \equiv g_Ix \pmod{m}$. Therefore $v$ is in the cycle of $v$.
\end{proof}

The previous proposition classifies exactly which vertices in $C $ will be tails. We next want to classify when a component will have tails. In order to proceed, we need the following lemma.

\begin{lemma} If $g_I$ is not squarefree, then $C_I$ must contain at least one tail. In particular, $\pi_I$ is a tail. \label{tail_exist}\end{lemma}

\begin{proof}
	By Propositon \ref{tail_struc}, we know that if $g_I$ is not squarefree, then tails may exist. Recall that the component $C_I$ must contain all elements divisible by the primes in $\pi_I$ and not divisible by all other primes in the factorization of $m$. Therefore the smallest element of $C_I$ is $\pi_I$. Since $g_I$ is not squarefree, then $g_I \nmid \pi_I$. Therefore $\pi_I$ is a tail.
\end{proof}

\begin{prop} 
	$ g_I$ is squarefree if and only if $C_I$ has no tails.\end{prop}
\begin{proof}
	By Lemma \ref{tail_exist}, we know that if $g_I$ is not squarefree, then tails exist. Thus, by the contrapositive, if $C_I$ has no tails, then $g_I$ is squarefree.
	
	Now suppose that $g_I$ is squarefree. Recall that the component $C_I$ must contain all elements divisible by the primes in $\pi_I$ and not divisible by all other primes in the factorization of $m$. Since $g_I$ is squarefree, then $g_I$ must be the smallest element in $C_I$; that is, $g_I = \pi_I$. Therefore all elements $v \in C_I$ are of the form $g_Ix $ for some $x \in \Z$.  Since $g_I | v$ for each $v \in C_I$, then by Proposition \ref{tail_struc}, $v$ cannot be a tail.
\end{proof}

The previous proposition fully classifies the components without tails since $g_I$ is squarefree if and only if all exponents $e_{i}$ for indices $i \in I$ are equal to 1. Recall that by Theorem \ref{idem_struc}, this occurs if and only if all exponents $e_{i} = 1$ in the factorization of $m$. 

\begin{theorem} \cite[Theorem 79]{porubsky}
	The sequential power graph of $\Zm$ has a connected component with a tail if and only if $m$ is not squarefree.
\end{theorem}

\subsubsection{Structure of the tails}

Consider a component $C_I$ with idempotent $d_I$ and multiplier $\pi_I$ that contains tails. Recall that $d_IU \subset C_I$ and observe that $d_IU = \{g_Ix \pmod{m} \, : \,  \text{gcd}(x, \frac{m}{g_I}) = 1\}$ is the set of non-tails. We can further describe the tails in terms of $U$. There may not be one tail $t$ such that $tU$ describes all the tails of $C_I$, but we can describe a set of $t_i$ such that $\{t_iU\}$ partitions the set of tails.

\begin{remark}
	If $I = R$, so $0 \in C_R$, then every element $x \in C_R- \{0\}$ is a tail since $x$ is nilpotent. There are $p^{e_1-1} \cdots p^{e_r-1}$ nilpotent elements, so there are $p^{e_1-1} \cdots p^{e_r-1} - 1$ tails in $C_R$. 
\end{remark}

With the nilpotent component classified, we now focus on the components $C$ not containing 0.

\begin{prop} Let $I \subsetneq R$. That is, let $\pi_I \ne 0$ be a multiplier. Let $y$ be a proper factor of $\frac{g_I}{\pi_I}$. Let $S =\{y\pi_I x \pmod{m} \,  : \, \text{gcd}(x, \frac{m}{y\pi_I}) = 1 \}$, considering $S$ as a set of equivalence classes modulo $m$ with representative $y\pi_Ix \in [0,m)$. Then $y\pi_I U = S$. 
	\label{tail_unit_relation}
\end{prop}
\begin{proof}
	First consider $y \pi_I u$ for $u \in U$. Since $\text{gcd}(u,m) =1$ then $y \pi_I u \in S$. Note that $\text{gcd}(u,m) =1$  also implies $y \pi_I u \in C_I$. Further, $g_I \nmid y \pi_I u$ so $y \pi_I u $ is a tail. In fact, $S$ is the set of tails in $C$ such that $y \pi_I $ divides the tail but no $z \pi_I$ divides the tail for $z  >y$. 
	
	By construction, $|S| = \phi(\frac{m}{y \pi_I})$. The same argument that shows $|d_IU| = \phi(\frac{m}{g_I})$ will show that $|y \pi_I U| = \phi(\frac{m}{y \pi_I})$. Since $y \pi_I U \subseteq S$ and the two finite sets have the same cardinality, then $ y \pi_I U = S$. 
\end{proof}

\begin{corollary} 
	\label{tail_partition}
	Let $I \ne \emptyset, R$. That is, let $\pi_I \ne 0,1$. The sets $y \pi_I U$ such that $y$ is a proper factor of $\frac{g_I}{\pi_I}$ partition the tails of the component determined by $I$. 	
\end{corollary}

\begin{corollary}
	\label{num_tails}
	In the above situation, the sets $y\pi_I U$ partition the set of tails of component $C_I$. The number of tails is $$\sum_{\substack{y | \frac{g_I}{\pi_I}\\ y \ne \frac{g_I}{\pi_I}}} \phi\left(\frac{m}{y \pi_I}\right).$$
\end{corollary}

We now understand the structure of the tails and we have related the tails to the unit component $U$. Further, since $d_IU$ is the set of all non-tails, we can fully determine the structure of the component $C_I$. 

\begin{corollary} 
	The sets $d_IU$ and $y \pi_IU$ where $y$ is defined as above partition $C_I$. 	\end{corollary}

\begin{remark}
	This gives an identity
	$$
	\frac{\phi(m)}{\phi(\pi_I)} = \sum_{y | \frac{g_I}{\pi_I}} \phi\left(\frac{m}{y\pi_I}\right).
	$$
\end{remark}

\begin{theorem} 
	For simplicity of notation, suppose $\pi= p_1 \cdots p_s$ and let $y | \frac{g}{\pi}$. Say $\pi y = p_{1}^{k_1} \cdots p_{s}^{k_s}$. Call $T_{\pi y}$ the set of tails of the form $\pi y u$ for $u \in U$. Recall that $T_{\pi y}$ is a set of equivalence classes modulo $m$. Then 
	$$
	T_{\pi y} \cong p_{1}^{k_1} (\Z / p_{1}^{e_1}\Z)^\times \cdots \times p_s^{k_s} (\Z / p_s^{e_s}\Z)^\times \times (\Z / p_{s+1}^{e_{s+1}}\Z)^\times \cdots (\Z / p_r^{e_r}\Z)^\times.
	$$
\end{theorem}
\begin{proof}
	This follows as in the proof of Theorem \ref{dU_struc}.
\end{proof}

We now understand both the explicit and algebraic structure of the sets $d_IU$ and $T_{\pi_I y}$. Therefore, we understand the structure of the connected component $C_I$. 
\begin{theorem}
	Let $C_I$ be a component determined by $I$. Then
	$$
	C_I \cong d_IU \cup \bigcup_{y | \frac{g_I}{\pi_I}, y \ne \frac{g_I}{\pi_I}} T_{\pi_I y}. 	
	$$
\end{theorem}
This recovers Theorem 3.9 in \cite{chakrabarty}. 

\subsubsection{Other propositions about tails}
\begin{prop}
	Let $e$ be the maximum exponent in the factorization of $d_I$. Then there are tails of length $e - 1$; that is, there exists a sequence $t \rightarrow t^2 \rightarrow \cdots \rightarrow t^{e-1}$ such that all elements in the sequence are tails.
\end{prop}
\begin{proof}
	Consider the tail $\pi_I$. Then $g_I\not | \pi_I^i$ for all $i =1, \ldots , e-1$. Therefore $\pi_I$ starts a tail of length $e-1$. 
\end{proof}

\begin{prop}
	Let $T_i = \{t \,: \, t \text{ is a tail, } \pi^i | t \text{ and } \pi^{i+1} \nmid t\}$. The tails form layers $T_1, T_2, \ldots, T_{e-1}$ so that tails in $T_i$ have length $\lceil \frac{e}{i} \rceil-1$. The sets $\{T_1, \ldots, T_{e-1}\}$ form another (possibly different) partition of the set of tails. 
\end{prop}
\begin{proof}
	Set $k= \lceil \frac{e}{i} \rceil$. It is clear that $g_I | t^{k}$ but $g_I$ does not divide $t^{k-1}$. Further, it is clear that $T_i \cap T_j = \emptyset$ when $i \ne j$ and that the sets $T_i$ cover all tails. 
\end{proof}

\begin{prop} Let $T_1^k  = \{t_j^k \, : \, t_j \in T_1 \}$. Then $T_1^k \subset T_k$.
\end{prop}
\begin{proof}
	Start with $t_j^k \in T_1^k$ for some $k < e$. Then $t_j^k \in T_k$ since $\pi^k | t_j^k$ but $\pi^{k+1}$ does not. 
\end{proof}

\subsubsection{Asymptotic behavior of the tails and the number of components}
In \cite{finchidem}, Finch considers the asymptotic behavior of certain sets of $\Zm$ relating to the sequence $x, x^2, x^3, \ldots$. These asymptotic expressions allow us to determine the average behavior of these sets, which allows us to better understand the average structure of the sequential power graphs. Though we have formulas for the number of elements in each set mentioned below for a fixed $m$, it is useful to understand the behavior of the graphs independent of $m$ as well.

Finch defines the set
$$
a(m) = \#\{x \in \Zm \, : \, x^{k+1} = x \text{ for some } k  \geq 1\}
$$
and finds that 
$$
\sum_{m\leq N} a(m) \sim A \cdot N^2
$$
where $A \approx .4408\ldots$ is a known constant.

Note that by definition, $a(m)$ is the number of non-tails in $\Zm$. Therefore on average, the number of non-tails in $\Zm$ is $A \cdot |\Zm|$, which implies that on average, the density of non-tails is $44.08\%$ and the density of tails is $55.92\%$.

We recall that for a fixed $m$, we can determine the exact number of tails using Corollary \ref{num_tails}. We obtain
$$
\sum_{I} \sum_{\substack{y | \frac{g_I}{\pi_I}\\ y \ne \frac{g_I}{\pi_I}}} \phi(\frac{m}{y \pi_I})
$$
tails. 
In \cite{finchsquares}, Finch gives an asymptotic formula for $|\Zm^\times|$ as follows
$$
\sum_{m \leq N} |\Zm^\times| = \sum_{m \leq N} \phi(m) \sim \frac{3}{\pi^2} N.
$$
This implies that on average, approximately $30.4\%$ of elements in $\Zm$ are units. 

Given a fixed $m$, we know there are exactly $\phi(m) = m (1- \frac{1}{p_1}) \cdots (1- \frac{1}{p_r})$ units. Also, given $m$, we can understand the number of elements in each component $C$ by using Corollary \ref{num_tails} and Proposition \ref{size_dU}. 

In the papers \cite{finchidem} and \cite{finchsquares}, Finch also finds that the average number of idempotents (and therefore the average number of components) in $\Zm$ is $\frac{6}{\pi^2} \ln(N)$, the average number of images of $y \mapsto y^2$ in $\Zm$ is approximately $.376 N (\ln N)^{-\frac{1}{2}}$, and the average number of images of $y \mapsto y^3$ in $\Zm$ is approximately $.484 N (\ln N)^{-\frac{1}{3}}$. Note that the last two asymptotics give an idea of how many elements should occur in each ``layer" of the sequential power graph. These papers and \cite{finch10} give several other interesting asymptotic formulas for various other sets involving exponentials.

\subsection{Summing elements in components}
In this section, we determine a closed form for the sum of elements in each component. Here we take the standard representatives $x \in \Z/ m\Z$ so that $0 \leq x < m$ and we treat these as elements in $\Z$. We proved closed formulas for both the sum of the elements in $d_IU$ and the sum of all elements in $C_I$. The first proposition is well-known.
\begin{prop} Let $U$ be the group of units in $\Z/ m\Z$. Then
	$$\sum_{u \in U} u = \frac{m\phi(m)}{2}.$$
\end{prop}
\begin{proof}
	This is clearly true if $m =2$. Now suppose $m > 2$. Observe that $u \in U$ if and only if $m-u \in U$. Since $\phi(m)$ is even when $m >2$, then $(u, m-u)$ must always come in pairs. We have $\frac{\phi(m)}{2}$ such pairs and $u + m- u = m$. 
\end{proof}

Now we show the formula for summing all elements in $d_I U$. 
\begin{prop} \label{dU_sum}
	If $d_I\ne 0$, then $\sum_{v \in d_IU} v = \frac{1}{\phi(g_I)}\sum_{u \in U} u = \frac{m\phi(m)}{2\phi(g_I)}$.
\end{prop}
\begin{proof}
	The homomorphism $H: U \rightarrow d_IU$ is a $\phi(g_I) : 1$ map by since $U \cong \left(\Z/ m\Z\right)^\times$ and $d_IU \cong \left(\Z / \frac{m}{g_I} \Z\right)^\times$. Therefore
	\begin{align*}
	\sum_{v \in d_IU} v  &= \sum_{v \in \left(\Z / \frac{m}{g_I} \Z\right)^\times} v\\
	&= \frac{1}{\phi(g_I)} \sum_{v \in \left(\Z/ m\Z\right)^\times }v\\
	&= \frac{1}{\phi(g_I)} \sum_{u \in U}u
	\end{align*}
\end{proof}

We now work toward proving a closed formula for the sum of all elements in $C_I$. We begin by using Inclusion-Exclusion.
\begin{lemma}
	Fix a component $C_I$. Let $f_{J}$ denote the number of prime divisors in $\frac{\pi_J}{\pi_I}$ when $\pi_I| \pi_J$ and let $M_J = \frac{m}{\pi_J}$. Then the sum of the elements in component $C_I$ is given by
	$$
	\sum_{I \subseteq J\subseteq R} (-1)^{f_J} \pi_J \binom{M_J}{2} 
	$$
\end{lemma}
\begin{proof}
	This follows from Inclusion-Exclusion. Recall that elements in $C_I$ are of the form $\pi_I x$ where $\text{gcd}(x,\frac{m}{g_I}) =1$. We first consider the sum of all multiples of $\pi_I$ given by $\pi_I(1 + 2 + \cdots + (M_I-1))$. We then subtract all elements that are divisible by one additional prime dividing $m$. We add back all elements divisible by two additional primes dividing $m$ and so on.
\end{proof}

\begin{remark}
	Note that when $m$ is squarefree, $\pi_R = m$ implies that $M_R = 1$.  Therefore $\binom{M_R}{2} =0$. 
\end{remark}

\begin{prop} \label{reduced_sum_form}
	Fix a component $C_I$ in the sequential power graph of $\Z/ m\Z$. Let $f_{J}$ denote the number of prime divisors in $\frac{\pi_J}{\pi_I}$ when $\pi_I| \pi_J$ and let $M_J = \frac{m}{\pi_J}$. Then the sum of the elements in component $C_I$ is given by
	
	$$
	\frac{m^2}{2} \sum_{I\subseteq J \subseteq R} \frac{(-1)^{f_J} }{\pi_J} .
	$$
	
\end{prop}
\begin{proof}
	We first plug in the formula for $M_J$ to obtain
	\begin{align*}
	\sum_{I\subseteq J \subseteq R}(-1)^{f_J} \pi_J \binom{M_J}{2} &= \sum_{I\subseteq J \subseteq R} (-1)^{f_J} \pi_J \cdot\frac{1}{2} \cdot \frac{m}{\pi_J} \left( \frac{m}{\pi_J}  - 1\right)\\
	&= \sum_{I\subseteq J \subseteq R} (-1)^{f_J} \pi_J \cdot\frac{1}{2} \cdot \left(\frac{m^2}{\pi_J^2} - \frac{m}{\pi_J} \right)\\
	&= \sum_{I\subseteq J \subseteq R}(-1)^{f_J} \pi_J \cdot\frac{1}{2} \cdot \frac{m^2}{\pi_J^2} - \sum_{I\subseteq J \subseteq R}(-1)^{f_J} \pi_J \cdot \frac{1}{2} \cdot \frac{m}{\pi_J} \\
	&= \frac{m^2}{2} \sum_{I\subseteq J \subseteq R} \frac{(-1)^{f_J} }{\pi_J} -\frac{m}{2} \sum_{I\subseteq J \subseteq R} (-1)^{f_J}
	\end{align*}
	Let $\#I = k$ and $\# R = r$. Then there are $\binom{r-k}{i}$ ways to pick $J$ such that $I \subset J$ and $\#(J \setminus I) = i$. Therefore we see that
	$$
	\sum_{I\subseteq J \subseteq R}(-1)^{f_J} = \sum_{i=0}^{r-k} (-1)^{i} \binom{r-k}{i} =0
	$$
	since the alternating sum of binomial coefficients is equal to 0. The formula follows.
	
\end{proof}

The following lemma by T\'oth in \cite{toth} allows us to reduce the formula from Proposition \ref{reduced_sum_form} further.
\begin{lemma} \cite{toth}\label{alt_sum_divs}
	Let $\beta(n)$ denote the alternating sum of divisors of a squarefree integer $n$. That is,
	$$
	\beta(n) = \sum_{d | n} d \lambda\left({\frac{n}{d}}\right)
	$$
	where $\lambda(k) = (-1)^{\Omega(k)}$ is the Liouville function and $\Omega(k)$ denotes the number of prime divisors of $k$.  Then
	$$
	\beta(n) = \phi(n).
	$$
\end{lemma}
\begin{remark}
	This is not an equality when $n$ is not squarefree, though it does prove a lower bound on $\beta(n)$. 
\end{remark}
\begin{prop} \label{alt_sum_mults}
	$\sum_{I\subseteq J \subseteq R} \frac{(-1)^{f_J}}{\pi_J} = \frac{\phi(\pi_{R\setminus I})}{\pi_R}.$
\end{prop}

\begin{proof}
	Consider
	$$
	\sum_{I\subseteq J \subseteq R} \frac{(-1)^{f_J}}{\pi_J}  = \frac{1}{\pi_I} + \sum_{i=1}^{r-k} (-1)^i \sum_{I\subseteq J \subseteq R} \frac{1}{\pi_J}
	$$
	just by expanding out the first summation. Now, we can get a common denominator by taking
	\begin{align*} \frac{1}{\pi_I} + \sum_{i=1}^{r-k} (-1)^i \sum_{I\subseteq J \subseteq R} \frac{1}{\pi_J} &=	\frac{1}{\pi_I} \left(1+  \sum_{i=1}^{r-k} \frac{(-1)^i}{\pi_{R \setminus I}} \sum_{\substack{J\subseteq R \setminus I\\ |J| = r-k-i}} \pi_J\right)\\
	&= 	\frac{1}{\pi_I} \left( \frac{1}{\pi_{R \setminus I}} \sum_{i=0}^{r-k} (-1)^i\sum_{\substack{J\subseteq R \setminus I\\ |J| = r-k-i}} \pi_J\right)\\
	\end{align*}
	Observe that this sum is simply the alternating sum of the divisors of $\pi_{R\setminus I}$, call it $\beta(\pi_{R\setminus I})$. By the Lemma \ref{alt_sum_divs}, $\beta(\pi_{R\setminus I}) = \phi(\pi_{R \setminus I})$. Therefore we have
	$$
	\sum_{I\subseteq J \subseteq R}\frac{(-1)^{f_J}}{\pi_J}  = \frac{1}{\pi_I} \cdot \frac{\phi(\pi_{R\setminus I})}{\pi_{R\setminus I}} = \frac{\phi(\pi_{R\setminus I})}{\pi_R}.
	$$
\end{proof}
By putting Proposition \ref{reduced_sum_form} and Proposition \ref{alt_sum_mults} together, we arrive at the following theorem.
\begin{theorem}\label{closed_sum}
	Fix a component $C_I$ in the sequential power graph of $\Z/ m\Z$. The sum of the elements in component $C_I$ is given by
	$$
	\frac{m^2}{2} \cdot \frac{\phi(\pi_{R\setminus I})}{\pi_R}.
	$$
\end{theorem}

\begin{corollary}
	The sum of the tail elements in $C_I$ is given by
	$$
	\frac{m^2}{2} \cdot \frac{\phi(\pi_{R\setminus I})}{\pi_R} - \frac{m}{2} \cdot \phi(g_{R \setminus I}).
	$$
\end{corollary}
\begin{proof}
	Apply Proposition \ref{dU_sum} and Proposition \ref{closed_sum}.
\end{proof}

\section{Homomorphisms between connected components}

\label{homomorphism}
We begin by giving a homomorphism from $U$ to $dU$. We note that Theorem \ref{homomorphism} generalizes the following proposition, but we present this here to show its simplicity. The homomorphism in Proposition \ref{group} was noted abstractly in \cite{schwarz}, but not explicitly given. We later extend this homomorphism to $d_I U \rightarrow d_K U$ whenever the components $C_I$ and $C_K$ are comparable.
\begin{prop} \cite[Theorem 2.3]{schwarz} \label{group}
	$f:U \rightarrow dU$ is a group homomorphism, where $f(u)=du$.
\end{prop}

\begin{proof}
	Let $du, dv \in dU$. Then $$(du)(dv) = d^2 uv \equiv duv \equiv d (uv) \pmod{m}.$$ Since $uv \in U$ as $U$ is a group, then $duv \in dU$. 
\end{proof}

\begin{prop}
	$dU$ is isomorphic to a subgroup of $U$.
\end{prop}		
\begin{proof}
	$U$ is the direct product of finite cyclic groups. As such, any homomorphic image of $U$ is isomorphic to a subgroup of $U$. 
\end{proof}

\begin{theorem} \label{lattice}
	The set of connected components $\{C_I\}$ forms a complete lattice where $C_{I} \leq C_{J}$ if and only if $\pi_I | \pi_J$. In other words, $C_I \leq C_J$ if and only if $I \subseteq J$. 
\end{theorem}
\begin{proof}
	Let $\pi_I$ and $\pi_J$ be multipliers of components $C_{I}$ and $C_{J}$ respectively. We define $C_{I} \vee C_{J} = \text{lcm}(\pi_I, \pi_J)$ and $C_{I} \wedge C_J = \text{gcd}(\pi_I, \pi_J)$. Since the multipliers are squarefree, then the supremum and infimum are defined for all pairs of components. 
\end{proof}

\begin{remark}
	This lattice is isomorphic to the lattice of idempotents of $\Zm$. This is a bounded lattice where the smallest element is $d =1$ and the largest element is $d=0$. The coatoms of the lattice are $d_{R \setminus \{i\}}$ and the atoms are $d_{\{i\}}$ for each $i = 1, \ldots, r$.
\end{remark}

\begin{theorem} \label{homomorphism}
	Let $C_I$ and $C_K$ be connected components in $\Zm$. Let $C_{I} \leq C_{K}$ (so $\pi_I | \pi_K$ and $I \subseteq K$).
	Note that $\pi_{K \setminus I} = \frac{\pi_K}{\pi_I}$.	
	Then $H:d_IU \rightarrow d_KU$ is a group homomorphism under multiplication, where $H(d_Iu)=d_{K \setminus I}d_Iu=d_Ku$.
\end{theorem}
\begin{proof}
	First consider $H((d_Iu)\cdot(d_Iv)) = H(d_Iuv) = d_{K \setminus I}d_Iuv=d_{K \setminus I}d_Iud_{K \setminus I}d_Iv = H(d_Iu) \cdot H(d_Iv)$, so $H$ is a homomorphism. 
	
	It is clear that $d_{K \setminus I}d_I$ is an idempotent. Further, observe that $p_i^{e_i} | d_{K \setminus I} d_I$ for all $i \in K \setminus I \cup I = K$. If $j \not \in K$, then $p_j^{e_j} $ cannot divide $d_{K \setminus I} d_I$ since $p_j^{e_j}$ does not divide $d_{K \setminus I}$ and it does not divide $d_I$. Therefore $d_{K \setminus I} d_I= d_K$. 
	So $d_{K \setminus I}d_I$ is the idempotent $d_K$ in $C_K$, and $H$ sends $C_I$ to $C_K$.
\end{proof}

\begin{theorem}
	Let $K \setminus I= \{j_1, \ldots, j_s\}$ and let $H$ be as in Theorem \ref{homomorphism}.
	The kernel of the homomorphism $H:d_IU \rightarrow d_KU$ is then isomorphic to $(\Z / p_{j_1}^{e_{j_1}} \Z)^\times \times  \cdots \times (\Z /p_{j_s}^{e_{j_s}}\Z)^\times$.
\end{theorem}
\begin{proof}
	Let $I = \{i_1, \ldots, i_t\}$ and $R \setminus K = \{h_1, \ldots, h_w\}. $ Then by Theorem \ref{homomorphism}, 
	\begin{align*}
	d_I &\cong (\Z/p_{j_1}^{e_{j_1}}\Z)^\times \times \cdots \times (\Z /p_{j_s}^{e_{j_s}}\Z)^\times \times (\Z / p_{h_1}^{e_{h_1}}\Z)^\times \times  \cdots \times (\Z / p_{h_w}^{e_{h_w}}\Z)^\times\\
	d_K& \cong  (\Z / p_{h_1}^{e_{h_1}}\Z)^\times \times  \cdots \times (\Z /p_{h_w}^{e_{h_w}}\Z)^\times
	\end{align*}
	Thus $H$ is a projection map, so the kernel is given by
	$$
	\text{Ker}(H) \cong (\Z / p_{j_1}^{e_{j_1}}\Z)^\times \times  \cdots \times (\Z / p_{j_s}^{e_{j_s}}\Z)^\times.$$
\end{proof}


\begin{theorem}
	Let $I \subseteq K$. The set of elements of the kernel of $H:d_IU \rightarrow d_KU$ is $\{d_I u \,  : \, u \equiv 1 \pmod{\frac{m}{g_K}}\}.$
\end{theorem}
\begin{proof}
	Suppose that $H(d_Iu)= d_K$. Then by definition of $H$, $ d_{K \setminus I} d_Iu \equiv d_Ku \pmod{m}$, which implies $d_K u \equiv d_K \pmod{m}$. Recall that $d_K= g_K a_K$ where $a_K$ is relatively prime to $\frac{m}{g_K}$. This congruence holds if and only if $a_K u \equiv a_K \pmod{\frac{m}{g_K}}$. Since $a_K$ is relatively prime to $\frac{m}{g_K}$, we obtain $u \equiv  1 \pmod{\frac{m}{g_K}}$. All of these implications are reversible, so the result follows.
\end{proof}

Recall that $d_K = a_K g_K$ where $g_K \equiv 0 \pmod{p_k}$ for all $k \in K$ and $d_K \equiv 1 \pmod{\frac{m}{g_K}}$. Note that $a_K$ is not unique. Thus it this kernel makes sense intuitively -- the elements that map to $d_K$ are the elements $g_I a_Iu$ such that $a_Iu$ is congruent to such a choice of $a_K$. 

By an almost identical proof, we can also classify the inverse images of elements in $d_KU$ under the map $H$. 

\begin{prop} \label{inverse_map}	Let $I \subseteq K$. The map $H: d_I U \rightarrow d_KU$ is a $\frac{\phi(g_K)}{\phi(g_I)}: 1$ map with $H^{-1}(d_K u) = \{ d_I v \, : \, v \equiv u \pmod{\frac{m}{g_K}} \text{ and } v \text{ is a unit}\}$.
\end{prop}

\begin{proof}
	Suppose that $H(d_Iu) = H(d_Iv)$. Then $d_Ku \equiv d_Kv \pmod{m}$, which implies that $u \equiv v \pmod{\frac{m}{g_K}}$. Observe that both $u$ and $v$ are units. To determine how many elements map to the same item, note that there are $\phi(\frac{m}{g_I})$ elements in $d_I U$ and there are $\phi(\frac{m}{g_K})$ elements such that $v \equiv u\pmod{\frac{m}{g_K}}$. Thus, this map is $\frac{\phi(g_K)}{\phi(g_I)} : 1$.
\end{proof}



\begin{remark}
	We have shown that each component is uniquely determined by a subset $I$ of $R$. Further, we have fully determined the structure of $C_I$ by understanding $d_IU$ and $y\pi_IU$ for $y$ and proper factor of $ \frac{g_I}{\pi_I}$. By the work in this section, the idempotent elements form a monoid under multiplication and we can determine the lattice structure of the components simply by knowing $I$. This implies that the entire structure of the sequential power graph is determined by the multiplicative structure of the idempotent elements and by the unit component $U$. 
\end{remark}

\section{Future directions}
Let $S$ be a multiplicative semigroup. We define the sequential power graph $(V,E)$ to be the graph with vertex set $V = S$ and with $(a,b) \in E$ if and only if $a = c^i$ and $b = c^{i+1}$ for some $i \in \N$ and $c \in S$. We have seen that the structure of the sequential power graph with $V = \Z/m\Z$ can be completely classified and this classification depends on the factorization of $m$ only. More generally if $S$ is finite, the basic structure of the components remains the same; that is, if $V = S$, then Theorems \ref{one_idem} through \ref{idem_comp} still hold. Thus every component will be uniquely determined by an idempotent in $S$ and all elements in a component $C_d$ will be roots of $d$.

Further, if multiplication is commutative, then Propositions \ref{dU_connected}, \ref{dU_mult_group},  \ref{one_idem}, and  \ref{group} hold as well. Thus when $S$ is a finite commutative semigroup and $U$ is the set of units in $S$, then $dU$ is a group, the components partition into $dU$ and a set of tails, and $f: U \rightarrow dU$ is a group homomorphism.

It is natural to ask what structure theorems one could prove about the sequential power graph for general finite $S$. If $S$ is commutative, what is the structure of $dU$ and what are the tails? If $S$ is not commutative, is there a way to describe the elements in each component? Do the components naturally partition into sets that are easier to understand?

\section{Acknowledgements}
A portion of this work was performed at the Ultrascale Systems Research Center (USRC) at Los Alamos National Laboratory, supported by the U.S. Department of Energy contract DE-FC02-06ER25750. The first author was supported in part by an appointment with the National Science Foundation (NSF) Mathematical Sciences Graduate Internship (MSGI) Program sponsored by the NSF Division of Mathematical Sciences. 

\bibliographystyle{plain}
\bibliography{biblio}

\end{document}